\documentclass[12pt]{extarticle}
\usepackage{extsizes}
\usepackage[reqno]{amsmath} 
\usepackage{amsmath}
\usepackage{amssymb}
\usepackage{amsthm}
\usepackage{amscd}
\usepackage{amsfonts}
\usepackage{graphicx}
\usepackage{dsfont}
\usepackage{fancyhdr}
\usepackage{enumerate}
\usepackage{youngtab}
\usepackage[utf8]{inputenc}
\usepackage{comment}

\usepackage{subfigure}
\usepackage[margin=1.5in]{geometry}
\usepackage{graphicx}
\usepackage{algorithm}
\usepackage{algpseudocode}
\usepackage{color}
\usepackage{hyperref}
\usepackage{notation}

\theoremstyle{theorem}

\newtheorem{corollary}{Corollary}

\newtheorem{lemma}[corollary]{Lemma}
\newtheorem{lemma*}[lem6]{Lemma}

\newtheorem{proposition}[corollary]{Proposition}
\newtheorem{theorem}[corollary]{Theorem}

\begin{document}

\AtEndDocument{%
  \par
  \medskip
  \begin{tabular}{@{}l@{}}%
    \textsc{Gabriel Coutinho}\\
    \textsc{Dept. of Computer Science} \\ 
    \textsc{Universidade Federal de Minas Gerais, Brazil} \\
    \textit{E-mail address}: \texttt{gabriel@dcc.ufmg.br} \\ \ \\
    \textsc{Emanuel Juliano} \\
    \textsc{Dept. of Computer Science} \\ 
    \textsc{Universidade Federal de Minas Gerais, Brazil} \\
    \textit{E-mail address}: \texttt{emanuelsilva@dcc.ufmg.br}\\ \ \\
    \textsc{Thomás Jung Spier} \\
    \textsc{Dept. of Computer Science} \\ 
    \textsc{Universidade Federal de Minas Gerais, Brazil} \\
    \textit{E-mail address}: \texttt{thomasjung@dcc.ufmg.br}
  \end{tabular}}

\title{No perfect state transfer in trees with more than 3 vertices}
\author{Gabriel Coutinho\footnote{gabriel@dcc.ufmg.br --- remaining affiliations in the end of the manuscript.} \and Emanuel Juliano \and Thomás Jung Spier}
\date{\today}
\maketitle
\vspace{-0.8cm}

\begin{abstract} 
    We prove that the only trees that admit perfect state transfer according to the adjacency matrix model are $P_2$ and $P_3$. This answers a question first asked by Godsil in 2012 and proves a conjecture by Coutinho and Liu from 2015.
\end{abstract}

\begin{center}
\end{center}
\begin{center}
\end{center}


\section{Introduction}\label{intro}

Let $G$ be a graph and $A$ its adjacency matrix. The graph is used to model a (continuous-time) quantum walk: each vertex represents a qubit and edges represent interaction. Once a particular type of initial state is set in the system, a time-dependent evolution governed by Schr\"{o}dinger's equation takes place, and this evolution is given by the matrix
\[
	\exp(\ii t A),
\]
with $t \geq 0$. Quantum walks have been extensively studied, both their theoretical applications \cite{QwalkReview2022} and experimental runs \cite{ImplementationQwalkIBM}, and can be used to implement quantum algorithms and quantum information protocols \cite{ChildsUniversalQComputation}. This paper concerns one of these protocols known as perfect state transfer. As the name suggests, it consists of initializing the quantum network with all qubits at the same state, then inputting an orthogonal state to a particular qubit, and observing after some time an evolution that allows for a complete recovery of that special state somewhere else in the network. This problem has been extensively studied by several communities \cite{CoutinhoGodsilSurvey}. Most theoretical results typically fall into two classes: either new constructions of families of graphs admitting perfect state transfer are found, or necessary conditions are exploited to show natural impediments or that it does not occur  for some graphs. This paper falls into the second category.

In 2012, Godsil \cite{GodsilStateTransfer12} asked if perfect state transfer occurs in a tree. Back then the question was motivated perhaps by curiosity, but it has since picked up relevance because the capability of building a network with few edges in which state transfer occurs at large combinatorial distances has been well motivated \cite{KayLimbo}. Trees are, naturally, the obvious candidates to be studied. The paths on 2 and 3 vertices admit perfect state transfer, but no other. Some paths of arbitrarily long length admit pretty good state transfer, an approximation of perfect state transfer, but at the cost of increasingly long (and difficult to find) waiting times. In 2015 Coutinho and Liu \cite{CoutinhoLiu2} proved that in the Laplacian matrix model, that in which the quantum walk is given by $\exp(\ii t L)$, there is no perfect state transfer in trees but for $P_2$. They also showed that if a tree (again, except for $P_2$) contains a perfect matching, then there is no perfect state transfer according to the adjacency matrix model, and conjectured that the hypothesis on the matchings is not necessary for trees with at least four vertices. In this paper we prove this conjecture.

\section{Preliminaries} \label{sec:prelim}

For the purposes of this paper, we can treat the topic in a pure mathematical formalism. Let $G$ be a graph, with vertices $i$ and $j$. We use the bra-ket notation: $\ket i$ represents the characteristic vector of a vertex, sometimes denoted by $e_i$, and $\bra i$ its conjugate transpose.  Perfect state transfer from $i$ to $j$ is defined as there existing $t > 0$ and $\lambda \in \Cds$ so that
\[
	\exp(\ii t A) \ket i = \lambda \ket j.
\] 

\subsection{Strong cospectrality}

Using the spectral decomposition of $A$ into a sum of distinct eigenvalues times each eigenprojector, say denoted by $A(G) = \sum \theta_r E_r$, it is easy to show that perfect state transfer implies $E_r \ket i = \pm E_r \ket j$ for all $r$. If this occurs, vertices $i$ and $j$ are called strongly cospectral. Those eigenvalues $\theta_r$ for which $E_r \ket i \neq 0$ are called the eigenvalue support of $i$. We will denote by $\Phi(G)$ the set of eigenvalues of a graph $G$, and for a vertex $i$ of $G$ we write $\Phi_i(G)$ (or simply $\Phi_i$) for its eigenvalue support. The characteristic polynomial of $A(G)$ (in the variable $t$) will be denoted by $\phi^G$. It is easy to see that if $i$ and $j$ are strongly cospectral, then for all $r$ we have $(E_r)_{ii} = (E_r)_{jj}$. This is known to be equivalent to $\phi^{G\setminus i}=\phi^{G\setminus j}$, in which case the vertices are called cospectral. For a lengthy discussion about strongly cospectral vertices and related concepts, we refer to \cite{godsil2017strongly}). To efficiently decide whether or not two vertices are strongly cospectral can be done by means of the following result.

\begin{theorem}[Corollary 8.4 in \cite{godsil2017strongly}]\label{thm:strcospec}
Vertices $i$ and $j$ of a graph $G$ are strongly cospectral if and only if $\phi^{G \setminus i} = \phi^{G \setminus j}$ and all poles of $\phi^{G\setminus \{i,j\}}/\phi^G$ are simple.
\end{theorem}

It is easy to see that if there is an automorphism of the graph $G$ that maps the vertex $i$ to the vertex $j$, then $i$ and $j$ are cospectral, however the same does not hold for strong cospectrality. Note also that as a consequence of Theorem~\ref{thm:strcospec}, if $i$ and $j$ are in distinct connected components of the graph $G$, then $i$ and $j$ are not strongly cospectral.

\subsection{Perfect state transfer} \label{sec:pst}

Deciding whether or not perfect state transfer occurs can be done in polynomial time \cite{CoutinhoGodsilPSTpolytime} by an application of the following characterization theorem.
\begin{theorem}[Theorem 2.4.4 in \cite{CoutinhoPhD}]
	 \label{thm:pstcha}
	Let $G$ be a graph, $A(G) = \sum \theta_r E_r$ the spectral decomposition of its adjacency matrix, and let $i,j \in V(G)$. There is perfect state transfer between $i$ and $j$ at time $t$ if and only if the following conditions hold.
	\begin{enumerate}[(a)]
		\item $E_r \ket i = \sigma_r E_r \ket j$, with $\sigma_r \in \{-1,+1\}$ (that is, $i$ and $j$ are strongly cospectral.)
		\item There is an integer $a$, a square-free positive integer $\Delta$ (possibly equal to 1), so that for all $\theta_r$ in the support of $i$, there is $b_r$ giving
			\[
				\theta_r = \frac{a + b_r \sqrt{\Delta}}{2}.
			\]
			In particular, because $\theta_r$ is an algebraic integer, it follows that all $b_r$ have the same parity as $a$.
		\item There is $g \in \Zds$ so that, for all $\theta_r$ in the support of $i$, $(b_0 - b_r)/g = k_r$, with $k_r \in \Zds$, and ${k_r \equiv (1-\sigma_r)/2 \pmod 2}$.
	\end{enumerate}
	If the conditions hold, then the positive values of $t$ for which perfect state transfer occurs are precisely the odd multiples of $\pi/(g \sqrt{\Delta})$.
\end{theorem}

Condition (b) above is due to Godsil \cite{GodsilPerfectStateTransfer12} and forces a strong necessary condition: the distinct eigenvalues in the support of vertices involved in perfect state transfer have to be separated by at least $1$. 

Recently we showed a slight yet quite necessary strengthening of this last observation in~\cite{coutinho2023decorated}:

\begin{theorem}[Theorem 14 in~\cite{coutinho2023decorated}] \label{thm:bound_PST}
	Assume perfect state transfer occurs in $G$ at minimum time $t$. Then $t \leq \pi/\sqrt{2}$, and therefore either $g$ or $\Delta$ (as in Theorem~\ref{thm:pstcha}) are $\geq 2$. 
\end{theorem}

\subsection{Interlacing polynomials}\label{sec:interlacing}

In this subsection, we recall two basic facts about interlacing polynomials that will be used in our proofs. Let $p$ and $q$ be real polynomials with real zeros. Assume that $p$ and $q$ have $n$ and $n-1$ distinct zeros, respectively. We say that $p$ and $q$ \textit{interlace} if between every two zeros of $q$ there is a zero of $p$ and vice-versa. We say that $p$ and $q$ \textit{strictly interlace} if they interlace and do not share zeros.

\begin{lemma}\label{lem:interlacing_1}[p. 150 in~\cite{GodsilAlgebraicCombinatorics}] The polynomials $p$ and $q$ interlace if, and only if, 
\[\dfrac{q}{p}(t)=\displaystyle\sum_{\ell=1}^n\dfrac{\lambda_\ell}{t-s_\ell},\] 
where $s_\ell$ are the zeros of $p$ and $\lambda_\ell\geq 0$ for every $\ell$. Furthermore, $p$ and $q$ strictly interlace if, and only if, $\lambda_\ell> 0$ for every $\ell$.
\end{lemma}

There is also the following characterization of interlacing when we instead consider the quotient $p/q$. It is quite straightforward from the above, but we add its proof for clarity.

\begin{lemma}\label{lem:interlacing_2} The polynomials $p$ and $q$ interlace if, and only if, 
\[\dfrac{p}{q}(t)=x-s-\displaystyle\sum_{\ell=1}^{n-1}\dfrac{\lambda_\ell}{t-s_\ell},\] where $s_\ell$ are the zeros of $q$ and $\lambda_\ell\geq 0$ for every $\ell$. Furthermore, $p$ and $q$ strictly interlace if, and only if, $\lambda_\ell> 0$ for every $\ell$. 
\end{lemma}
\begin{proof} We prove the second part of the statement, the first one is analogous. Assume that 
\[\dfrac{p}{q}(t) = t-s-\displaystyle\sum_{\ell=1}^{n-1}\dfrac{\lambda_\ell}{t-s_\ell},\]
where $s_1<\dots<s_{n-1}$ are the zeros of $q$ and $\lambda_\ell>0$ for every $\ell$. Then 
\[\left(\dfrac{p}{q}\right)'(t)=1+\displaystyle\sum_{i=1}^{n-1}\dfrac{\lambda_\ell}{(t-s_\ell)^2}>1\] 
for every $t$ not in $\{s_1,\dots,s_{n-1}\}$. It follows that $(p/q)(t)$ is strictly increasing and has a unique zero in each of its $n$ branches $(-\infty, s_1)$, $(s_1,s_2)$, \dots, $(s_{n-1},\infty)$. These zeros are precisely the zeros of $p$. Thus $p$ and $q$ strictly interlace.

Now assume that $p$ and $q$ strictly interlace. Since $\deg p= \deg q+1$ we have $p(t)=(t-s)q(t)-b\cdot r(t)$, where $r$ is monic of degree $n-2$ and $b$ is real. In the proof of Lemma 5.1 in~\cite{GodsilAlgebraicCombinatorics} it is proved that in this situation $b>0$ and $q$ strictly interlaces $r$. Thus by Lemma~\ref{lem:interlacing_1} we can write 
\[\dfrac{r}{q}(t)=\displaystyle\sum_{i=1}^{n-1}\dfrac{\lambda_\ell}{t-s_\ell}\] where $s_\ell$ is a zero of $q$ and $\lambda_\ell>0$ for every $\ell$. Therefore 
\[\dfrac{p}{q}(t)= t-s-b\cdot \dfrac{r}{q}(t)=t-s-\displaystyle\sum_{\ell=1}^{n-1}\dfrac{b\lambda_\ell}{t-s_\ell}, \] where $b\lambda_\ell>0$ for every $\ell$. This finishes the proof.
\end{proof}

\subsection{Variational inequalities}

In this subsection we state an important result we will need. For a real symmetric $m \times m$ matrix $A$ we denote by $\theta_{m}(A)\leq \cdots \leq \theta_1(A)$ its eigenvalues. We also denote by 
\[\lVert A\rVert_{op}=\sup_{\braket{v}{v}=1}\lVert A\ket v \rVert=\max\{|\theta_1(A)|, |\theta_{m}(A)|\}\] its operator norm.


The next theorem is a standard result about eigenvalues of Hermitian matrices, but since we actually need the equality case, we provide a proof.

\begin{theorem}\label{thm:eigen_stability_2}[Corollary $III.2.6$ in~\cite{BhatiaMatrixAnalysis}] Let $A$ and $B$ be two real symmetric $m\times m$ matrices. Then, 
\[|\theta_\ell(A)-\theta_\ell(B)|\leq \lVert A-B\rVert_{op}, \]
for every $\ell$. If equality holds for $\ell$, then there is some vector $\ket v \neq 0$ in the subspace generated by the 1st and $m$-th eigenspaces of $A-B$ which is an $\ell$-th eigenvector of both $A$ and $B$.
\end{theorem}
\begin{proof} Consider $\ell$ in $\{1,\dots, m\}$. By the min-max principle~\cite[Corollary 111.1.2]{BhatiaMatrixAnalysis}, 
\[\theta_\ell(A)= \max_{\dim V = \ell}\,\min_{\substack{\ket v \in V \\ \lVert \ket v \rVert = 1 } } \bra v A \ket v\] and 
\[\theta_\ell(B)= \min_{\dim W = m-\ell+1}\,\max_{\substack{\ket v\in W\\ \lVert \ket v \rVert = 1 } } \bra v B \ket v.\]

Let $V$ and $W$ be subspaces of $\mathds{R}^{m}$ such that the maximum and minimum in the above equations are reached. Since $\dim V=\ell$ and $\dim W=m-\ell+1$, it holds $\dim V\cap W\geq 1$. Let $\ket v$ be a vector in $V\cap W$ with $\lVert \ket v \rVert = 1 $. Then, $\theta_\ell(A)\leq \bra v A \ket v$ and $\theta_\ell(B)\geq \bra v B \ket v$, and so 
\[\theta_\ell(A)-\theta_\ell(B)\leq \bra v A \ket v - \bra v B \ket v = \bra v (A-B) \ket v \leq \lVert A-B\rVert_{op}.\] 
Switching the roles of $A$ and $B$ we also have $\theta_\ell(B)-\theta_\ell(A)\leq \lVert A-B\rVert_{op}$. This proves the first part of the statement. 

If $\theta_\ell(A)-\theta_\ell(B) = \lVert A-B\rVert_{op}$, then notice that the vector $\ket v$ used in the inequalities above needs to be an $\ell$-th eigenvector of both $A$ and $B$ and be contained in the subspace generated by the $1$st and $m$-th eigenspaces of $A-B$. This proves the second part of the statement.
\end{proof}

\section{Main result}

In this section we state our main result and describe its proof strategy, leaving the details for the next section. 

Assume $i$ and $j$ are distinct vertices in $G$ with neighbors $i'$ and $j'$, respectively, such that $ii'$ and $jj'$ are cut-edges, each separating $i$ and $j$ into different connected components. We analyze two possibilities: either $i'=j$ and $j'=i$, and $ij$ is a cut-edge; or $i'\neq j$ and $j'\neq i$. In Figure~\ref{fig:graph_1} we illustrate these two possibilities.

\begin{figure}[H]
    \centering
    \includegraphics[width=14cm]{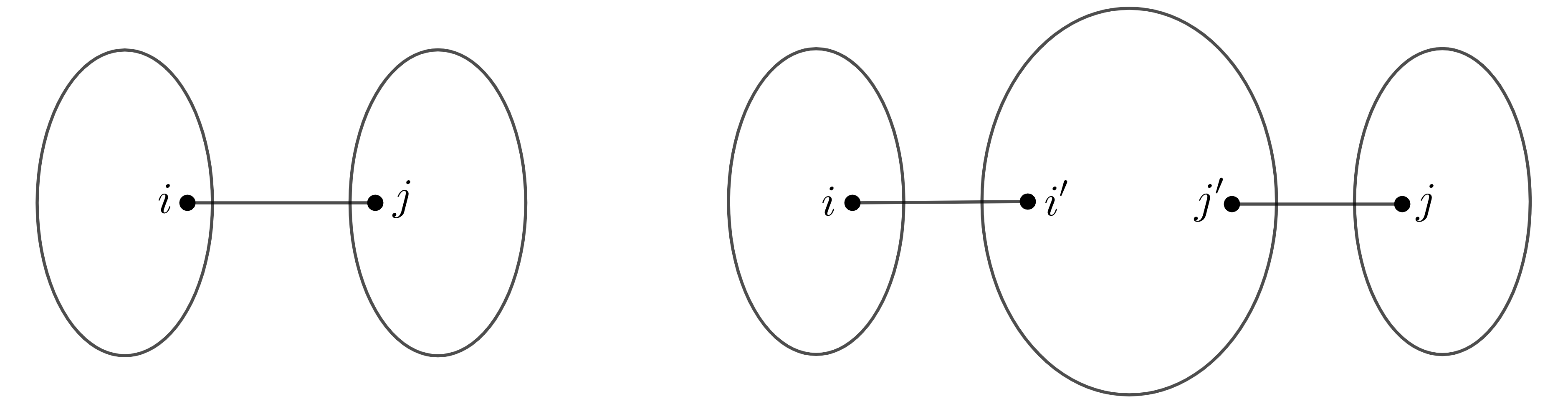}
    \caption{On the left side we represent a graph where $ij$ is a cut-edge. On the right side we illustrate a graph where the vertices $i$ and $j$ have neighbors $i'\neq j$ and $j'\neq i$, and $ii'$ and $jj'$ are cut-edges that each separates $i$ and $j$ into different connected components. Observe that it may be that $i'=j'$.}
    \label{fig:graph_1}
\end{figure}

The case in which $ij$ is a bridge, on the left side of Figure~\ref{fig:graph_1}, was taken care of in \cite[Theorem 11]{CoutinhoGodsilJulianovanBommel}, with a proof that the only graph which admits perfect state transfer between two vertices separated by a bridge is $P_2$. So we focus on the other case which remained open.

\begin{proposition}\label{prop:bound_trees_2} Let $G$ be a graph. Assume that $i$ and $j$ are strongly cospectral vertices in $G$ with neighbors $i'\neq j$ and $j'\neq i$, respectively, such that $ii'$ and $jj'$ are cut-edges that each separates $i$ and $j$ into different connected components. Then, the minimum distance between elements in $\Phi_i$ is smaller than, or equal to, $\sqrt{2}$, with equality if, and only if, $G$ is isomorphic to $P_3$.
\end{proposition}

It is immediate to verify that \cite[Theorem 11]{CoutinhoGodsilJulianovanBommel}, Proposition~\ref{prop:bound_trees_2} and Theorem~\ref{thm:bound_PST} imply our main result:

\begin{theorem}\label{thm:main_result} Let $G$ be a graph with more than three vertices.
\begin{enumerate}
	\item Assume that $i$ and $j$ are distinct vertices in $G$ with neighbors $i'$ and $j'$ respectively, such that $ii'$ and $jj'$ are cut-edges. Note that $i'$ and $j'$ are possibly equal to each other, or to $j$ and $i$, respectively.
	\item Assume the removal of $ii'$ and the removal of $jj'$ separate $i$ and $j$ into two connected components.
\end{enumerate}
Then, there is no perfect state transfer between $i$ and $j$.
\end{theorem}

Recall that both $P_2$ and $P_3$ admit perfect state transfer. Since the only connected trees with at most three vertices are $P_2$ and $P_3$ we obtain the following consequence.

\begin{theorem}\label{cor:main_corollary_1} No tree on more than three vertices admits perfect state transfer.
\end{theorem}

This answers a question by Godsil~\cite[Question (5)]{GodsilStateTransfer12} and by the first author~\cite[Problem 1]{CoutinhoPhD}, later stated as a conjecture in \cite[Conjecture 1]{CoutinhoLiu2} and in \cite[Conjecture 1]{coutinho2022quantum}. Corollary~\ref{cor:main_corollary_1} generalizes early results in the literature that there is no perfect state transfer on paths on more than three vertices~\cite{ChristandlPSTQuantumSpinNet2}, double stars~\cite[Theorem 4.6]{fan2013pretty} and trees on more than 2 vertices with unique perfect matchings \cite[Theorem 7.2]{CoutinhoLiu2}. By noticing that our argument generally holds for graphs with integer weights, a corollary of our Theorem~\ref{thm:main_result} is the result that trees on more than two vertices do not admit Laplacian perfect state transfer, proved in \cite[Corollary 4.4]{CoutinhoLiu2}.

Theorem~\ref{thm:main_result} also implies that the two main constructions of strongly cospectral pairs presented in the work of Godsil and Smith~\cite[Theorem 9.1 and Lemma 9.3]{godsil2017strongly} do not produce perfect state transfer, except in the trivial cases of $P_2$ and $P_3$. Our result also shows that there is no adjacency (or Laplacian) perfect state transfer between vertices of degree $1$, except in trivial cases.

\section{Strong cospectrality and rational functions}\label{sec:strategy}

To prove Proposition~\ref{prop:bound_trees_2}, we will make use of the machinery presented in this section. We begin with the following well known equality.

\begin{lemma}[Lemma 2.1 in \cite{GodsilAlgebraicCombinatorics}]\label{lem:wronskian} 
Let $i$ and $j$ be vertices in a graph $G$. Then,
\[\phi^{G\setminus i}\phi^{G\setminus j}-\phi^{G\setminus\{i,j\}}\phi^G=\left(\sum_{P:i\to j}  \phi^{G\setminus P}\right)^2.\]
\end{lemma}
The sum on the right hand side runs over all paths connecting $i$ to $j$. The previous lemma has the following immediate consequence for cospectral vertices.

\begin{corollary}\label{cor:cristoffel-darboux} Let $i$ and $j$ be cospectral vertices in the weighted graph $G$. Then, 

\[\left(\phi^{G\setminus i}-\sum_{P:i\to j} \phi^{G\setminus P}\right)\left(\phi^{G\setminus i}+\sum_{P:i\to j} \phi^{G\setminus P}\right)=\phi^G\phi^{G\setminus\{i,j\}}.\]
\end{corollary}
\begin{proof} Since $i$ and $j$ are cospectral, we have $\phi^{G\setminus i}=\phi^{G\setminus j}$. To finish the proof we simply apply Lemma~\ref{lem:wronskian} and rearrange the terms.
\end{proof}

We can also obtain explicit formulae for the entries of the eigenprojectors. See for instance \cite[Section 2.3]{CoutinhoGodsilJulianovanBommel} or \cite[Chapter 4]{GodsilAlgebraicCombinatorics}.

\begin{theorem}\label{thm:neutrino} Let $i$ and $j$ be vertices in a graph $G$. Then,
\begin{enumerate}[(a)]
    \item $\dfrac{\phi^{G\setminus i}}{\phi^G}(t) = \displaystyle\sum_{\theta \in \Phi(G)}\dfrac{(E_\theta)_{i,i}}{t-\theta};$
    \item $\dfrac{\sum_{P:i\to j} \phi^{G\setminus P}}{\phi^G}(t)=\displaystyle\sum_{\theta \in \Phi(G)}\dfrac{(E_\theta)_{i,j}}{t-\theta}.$
\end{enumerate} 
\end{theorem}

Note that $\theta$ is in $\Phi_i$ if and only if $(E_\theta)_{i,i}>0$. In other words $\Phi_i$ is precisely the set of poles of $\phi^{G\setminus i}/ \phi^G$. When two vertices $i$ and $j$ are strongly cospectral, we define $\Phi_{ij}^{+}$ to be the set of eigenvalues $\theta$ for which $E_\theta \ket i = E_\theta \ket j \neq 0$, and $\Phi_{ij}^{-}$ to be the set of eigenvalues $\theta$ for which $E_\theta \ket i = - E_\theta \ket j \neq 0$.

\begin{proposition}\label{prop:pos_neg_spectrum} Let $G$ be a graph and $i$ and $j$ be strongly cospectral vertices in $G$. Then, $\Phi_{ij}^{+}$ and $\Phi_{ij}^{-}$ are respectively the set of poles of 
\[\dfrac{\phi^{G\setminus i}+\sum_{P:i\to j} \phi^{G\setminus P}}{\phi^G} \quad \text{and} \quad \dfrac{\phi^{G\setminus i}-\sum_{P:i\to j} \phi^{G\setminus P}}{\phi^G}. \]
\end{proposition}
\begin{proof} By Theorem~\ref{thm:neutrino} we have that 
\[\dfrac{\phi^{G\setminus i}+\sum_{P:i\to j} \phi^{G\setminus P}}{\phi^G}(t)=\sum_{\theta \in \Phi_i}\dfrac{(E_\theta)_{i,i}+(E_\theta)_{i,j}}{t-\theta} = \sum_{\theta \in \Phi_{ij}^{+}}\dfrac{2(E_\theta)_{i,i}}{t-\theta}.\]
Thus $\theta$ is a pole of this rational function if and only if $\theta$ is in $\Phi_{ij}^{+}$. The same reasoning applies to $\Phi_{ij}^{-}$.
\end{proof}

It turns out that it is useful to take the inverse of these functions and treat $\Phi_{ij}^{+}$ and $\Phi_{ij}^{-}$ as the zeros of two rational functions.

Inspired by the notation in \cite{coutinho2022strong}, we define

\[\alpha_{ij}^{+}=\dfrac{\phi^{G\setminus i}-\sum_{P:i\to j} \phi^{G\setminus P}}{\phi^{G\setminus \{i,j\}}}, \quad \alpha_{ij}^{-}=\dfrac{\phi^{G\setminus i}+\sum_{P:i\to j} \phi^{G\setminus P}}{\phi^{G\setminus \{i,j\}}}.\] 

\begin{lemma}\label{thm:alphas_characterization} Let $i$ and $j$ be strongly cospectral vertices in a graph $G$. Then, $\Phi_{ij}^{+}$ and $\Phi_{ij}^{-}$ are precisely the set of zeros of $\alpha_{ij}^{+}$ and $\alpha_{ij}^{-}$, respectively. Furthermore, 

\[\alpha_{ij}^{+}(t) = t-r_0^+-\displaystyle\sum_{\ell=1}^{|\Phi_{ij}^{+}|-1}\dfrac{\lambda_\ell^+}{t-r_\ell^+},\quad \alpha_{ij}^{-}(t) = t-r_0^--\displaystyle\sum_{\ell=1}^{|\Phi_{ij}^{-}|-1}\dfrac{\lambda_\ell^-}{t-r_\ell^-},\]

\noindent where $\lambda_\ell^+, \lambda_\ell^-> 0$ for every $\ell$.
\end{lemma}
\begin{proof} Observe that, by Corollary~\ref{cor:cristoffel-darboux},

\[\alpha_{ij}^{+}=\dfrac{\phi^{G\setminus i}-\sum_{P:i\to j} \phi^{G\setminus P}}{\phi^{G\setminus \{i,j\}}} = \dfrac{\phi^G}{\phi^{G\setminus i}+\sum_{P:i\to j} \phi^{G\setminus P}}.\]

This last function is precisely the inverse of the function that 
appears in Proposition~\ref{prop:pos_neg_spectrum} whose poles are the elements of $\Phi_{ij}^{+}$. Therefore we conclude that $\Phi_{ij}^{+}$ is precisely the set of zeros of $\alpha_{ij}^{+}$.

We also know, from the proof of Proposition~\ref{prop:pos_neg_spectrum}, that

\[\dfrac{\phi^{G\setminus i}+\sum_{P:i\to j} \phi^{G\setminus P}}{\phi^G} = \sum_{\theta \in \Phi_{ij}^{+}}\dfrac{2(E_\theta)_{i,i}}{t-\theta}, \]

\noindent where $2(E_\theta)_{i,i}>0$ for every $\theta$ in $\Phi_{ij}^{+}$. Thus by Lemmas~\ref{lem:interlacing_1} and~\ref{lem:interlacing_2} we obtain that 
\[\alpha_{ij}^{+}(t) = t-r_0^+-\displaystyle\sum_{\ell=1}^{|\Phi_{ij}^{+}|-1}\dfrac{\lambda_\ell^+}{t-r_\ell^+}, \] 
where $\lambda_\ell^+> 0$, for every $\ell=1,\dots ,|\Phi_{ij}^{+}|-1$. 

The same reasoning applies to $\alpha_{ij}^{-}$.
\end{proof}

Notice that the difference between 
\begin{equation} \label{eq:difalpha}
	\alpha_{ij}^{+} - \alpha_{ij}^{-} = -2\dfrac{\sum_{P:i\to j}\phi^{G\setminus P}}{\phi^{G\setminus \{i,j\}}}.
\end{equation}

We will heavily exploit this fact below in our main result in this section. 

Let $i$ and $j$ be strongly cospectral vertices in a graph $G$. Assume that every path from $i$ to $j$ goes through the neighbors $i_1,\dots, i_u$ and $j_1,\dots, j_v$ of $i$ and $j$, respectively. Notice that the sets $\{i_1, \dots, i_u\}$ and $\{j_1, \dots, j_v\}$ may intersect. In Figure~\ref{fig:graph_2} we illustrate this situation.

\begin{figure}[H]
    \centering
    \includegraphics[width=12cm]{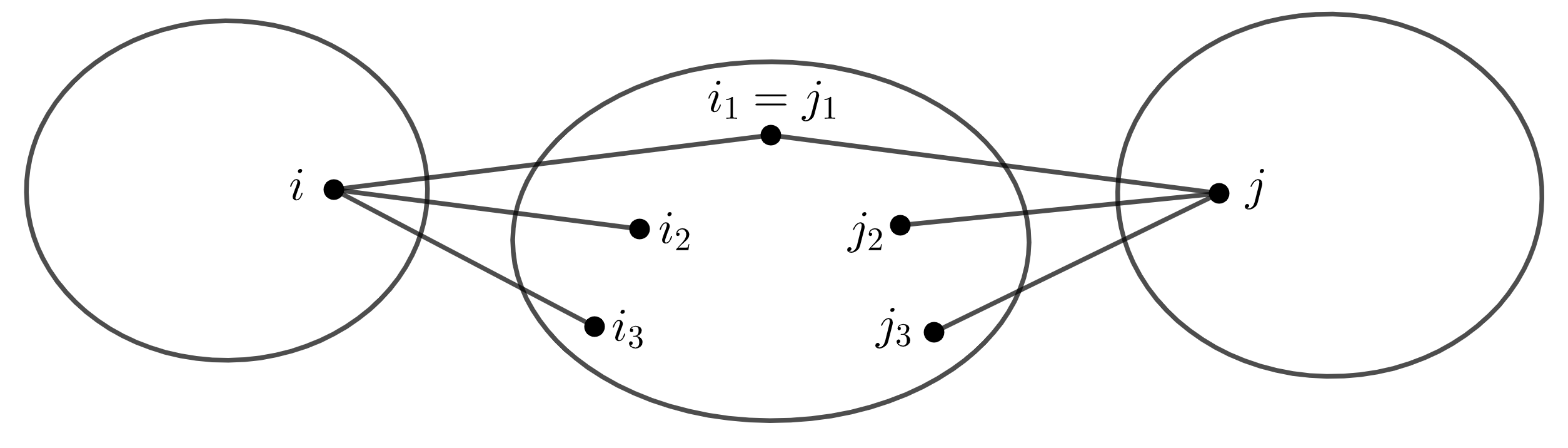}
    \caption{Observe that every path from $i$ to $j$ goes through the neighbors $i_1,i_2,i_3$ and $j_1,j_2,j_3$.}
    \label{fig:graph_2}
\end{figure}

\begin{lemma}\label{prop:partial_fraction} Let $i$ and $j$ be vertices in the graph $G$. Assume that every path from $i$ to $j$ goes through the neighbors $i_1,\dots, i_u$ and $j_1,\dots, j_v$ of $i$ and $j$, respectively. Then, 
\[\dfrac{\sum_{P:i\to j} \phi^{G\setminus P}}{\phi^{G\setminus \{i,j\}}}(t)=\displaystyle\sum_{\ell=1}^k\dfrac{\lambda_\ell}{t-s_\ell} , \] 
where $s_\ell, \lambda_\ell \in \mathds{R}$ for every $\ell$ and 
\[\displaystyle\sum_{\ell=1}^k|\lambda_\ell|\leq \sqrt{uv}.\]
\end{lemma}
\begin{proof}  

Let $H'$ be the graph formed by the connected components of $G\setminus \{i,j\}$ containing $i_1,\dots,i_u$ and $j_1,\dots, j_v$. Let $H$ be a graph formed by the remaining connected components of $G\setminus \{i,j\}$. Thus $G\setminus \{i,j\} = H'\sqcup H$.

Observe that there is a correspondence between paths from $i$ to $j$ in $G$ and paths from $\{i_1,\dots, i_u\}$ to $\{j_1,\dots, j_v\}$ in $H'$. For a path $P$ from $i$ to $j$ we write $P'$ for the corresponding path from $\{i_1,\dots, i_u\}$ to $\{j_1,\dots, j_v\}$ in $H'$, and vice-versa. Also, notice that $\phi^{G\setminus \{i,j\}}=\phi^{H'}\cdot \phi^H$, and that, for every path $P$ from $i$ to $j$, it holds $\phi^{G\setminus P} = \phi^{H'\setminus P'} \cdot \phi^H$. Thus, we obtain, 

\begin{align*}
	\dfrac{\sum_{P:i\to j}\phi^{G\setminus P}}{\phi^{G\setminus \{i,j\}}} & = \dfrac{\sum_{k=1}^u\sum_{\ell=1}^v\sum_{P':i_k\to j_\ell}\phi^{H'\setminus P'}\cdot \phi^H}{\phi^{G\setminus \{i,j\}}} 
	\\ & = \displaystyle\sum_{k=1}^u\displaystyle\sum_{\ell=1}^v \frac{\sum_{P':i_k\to j_\ell} \phi^{H'\setminus P'}}{\phi^{H'}}.
\end{align*}

By Theorem~\ref{thm:neutrino} we know that, for every $k$ and $l$,

\[\dfrac{\sum_{P':i_k\to j_\ell}\phi^{H'\setminus P'}}{\phi^{H'}} = \sum_{\theta \in \Phi(H')}\dfrac{(E_\theta)_{i_k,j_\ell}}{t-\theta},\]

\noindent where $E_\theta$ is the eigenprojector of $\theta$ for the graph $H'$. Write $\ket{\chi_i}$ for the vector indexed by the vertices of $H'$ which is the characteristic vector of the neighbors of $i$, and analogously consider the vector $\ket{\chi_j}$. It follows that, 

\[ \dfrac{\sum_{P:i\to j}\phi^{G\setminus P}}{\phi^{G\setminus \{i,j\}}}= \sum_{\theta \in \Phi(H')}\frac{\sum_{k=1}^u\sum_{\ell=1}^v (E_\theta)_{i_k,j_\ell}}{t-\theta}=\sum_{\theta \in \Phi(H')}\dfrac{\bra{\chi_i} E_\theta \ket{\chi_j}}{t-\theta}.\]

Finally, two applications of Cauchy-Schwarz and the fact that $\sum E_\theta = I$ give

\begin{align*}
	\sum_{\theta \in \Phi(H')}|\bra{\chi_i} E_\theta \ket{\chi_j}| & \leq \sum_{\theta \in \Phi(H')} \lVert E_\theta \ket{\chi_i} \rVert \cdot \lVert E_\theta \ket{\chi_j} \rVert \\ & \leq \sqrt{\sum_{\theta \in \Phi(H')}\lVert E_\theta \ket{\chi_i} \rVert^2}\sqrt{\sum_{\theta \in \Phi(H')}\lVert E_\theta \ket{\chi_j} \rVert^2} \\ & = \lVert \ket{\chi_i} \rVert \cdot \lVert \ket{\chi_j} \rVert \\ & = \sqrt{uv}.
\end{align*}
\end{proof}

We obtain an immediate corollary.

\begin{corollary}\label{cor:partial_fraction_2} Let $G$ be a graph and $i$ and $j$ be vertices in $G$ with neighbors $i'\neq j$ and $j'\neq i$, respectively, such that $ii'$ and $jj'$ are cut-edges that each separates $i$ and $j$ into different connected components. Then, 
\[\dfrac{\sum_{P:i\to j}\phi^{G\setminus P}}{\phi^{G\setminus \{i,j\}}}(t) = \displaystyle\sum_{\ell=1}^k\dfrac{\lambda_\ell}{t-s_\ell},\] where $s_\ell, \lambda_\ell \in \mathds{R}$ for every $\ell$ and $\displaystyle\sum_{\ell=1}^k|\lambda_\ell|\leq 1$.
\end{corollary}


\section{Proof of main result}\label{sec:strong_cospectrality}

Consider a rational function of the form 
\[\alpha(t)=t-s_0-\sum_{\ell=1}^k\dfrac{\lambda_\ell}{t-s_\ell}\]
where $\lambda_\ell\geq 0$ for every $\ell$. Observe that every zero of this rational function is a zero of 
\[\prod_{\ell=0}^k(t-s_\ell)-\sum_{\ell=1}^k\lambda_\ell\prod_{m \neq \ell}(t-s_m).\] 
This expression is the characteristic polynomial of the following matrix,

\[M_\alpha:=\left[ {\begin{array}{cccccccc}
	s_0  & \sqrt{\lambda_1} & \sqrt{\lambda_2} & \cdots  & \sqrt{\lambda_k}  \\
    \sqrt{\lambda_1} & s_1 & 0 & \cdots  & 0  \\
    \sqrt{\lambda_2} & 0 & s_2 & \cdots  & 0  \\
	\vdots   & \vdots   & \vdots   &  \ddots  &  \vdots  \\ 
	\sqrt{\lambda_k}   & 0 & 0 & \cdots  & s_k \\ 
	\end{array} } \right].\]

Recall the statement of Lemma~\ref{thm:alphas_characterization}. Let

\begin{equation} \label{eq:k}
	\{r_1,\dots, r_k\} = \{r_1^+,\dots, r_{|\Phi_{ij}^+|-1}^+\}\cup\{r_1^-,\dots, r_{|\Phi_{ij}^-|-1}^-\}.
\end{equation}

Observe that $k$ is smaller than, or equal to, $|\Phi_{ij}^+|+|\Phi_{ij}^-|-2$. We can write
\begin{equation} \label{eq:alphask}
	\alpha_{ij}^{+}(t) = t-r_0^+-\displaystyle\sum_{\ell=1}^{k}\dfrac{\mu_\ell^+}{t-r_\ell},\quad \alpha_{ij}^{-}(t) = t-r_0^--\displaystyle\sum_{\ell=1}^{k}\dfrac{\mu_\ell^-}{t-r_\ell},
\end{equation}
where $\mu_\ell^+, \mu_\ell^-\geq 0$ for all $\ell=1,\dots, k$, and either $\mu_\ell^+>0$ or $\mu_\ell^->0$ for every $\ell$. By our previous discussion we can consider the $(k+1)\times (k+1)$ matrices $M_{\alpha_{ij}^{+}}$ and $M_{\alpha_{ij}^{-}}$ associated with $\alpha_{ij}^{+}$ and $\alpha_{ij}^{-}$, respectively. To ease the notation we will refer to these matrices by $M_{ij}^{+}$ and $M_{ij}^{-}$.

From the $k+1$ eigenvalues of $M_{ij}^{+}$ and $M_{ij}^{-}$ there are exactly $|\Phi_{ij}^+|$ and $|\Phi_{ij}^-|$ zeros of $\alpha_{ij}^{+}$ and $\alpha_{ij}^{-}$, respectively.

We can translate the content of Corollary~\ref{cor:partial_fraction_2} to properties of the difference of the matrices $M_{ij}^{+}$ and $M_{ij}^{-}$. This is established in our next result.

\begin{lemma}\label{lem:matrix_difference_2} Let $G$ be a graph. Assume that $i$ and $j$ are strongly cospectral vertices in $G$ with neighbors $i'\neq j$ and $j'\neq i$, respectively, such that $ii'$ and $jj'$ are cut-edges that each separates $i$ and $j$ into different connected components. Then,

\[M_{ij}^{+} - M_{ij}^{-} = \left[ {\begin{array}{cccccccc}
	0 & \sqrt{\mu_1^+}-\sqrt{\mu_1^-} & \cdots & \sqrt{\mu_k^+}-\sqrt{\mu_k^-} \\
    \sqrt{\mu_1^+}-\sqrt{\mu_1^-} & 0 & \cdots & 0 \\
	\vdots & \vdots & \ddots  &  \vdots  \\ 
	\sqrt{\mu_k^+}-\sqrt{\mu_k^-} & 0 & \cdots & 0 
	\end{array} } \right],\]
 where $\sum_{\ell=1}^k \left(\sqrt{\mu_\ell^+}-\sqrt{\mu_\ell^-}\right)^2\leq \sum_{\ell=1}^k |\mu_\ell^+-\mu_\ell^-|\leq 2$.
\end{lemma}
\begin{proof} By Corollary~\ref{cor:partial_fraction_2} we have that 
\[\alpha_{ij}^{+}-\alpha_{ij}^{-}=\sum_{\ell=1}^k\dfrac{\mu_\ell^+-\mu_\ell^-}{t-r_\ell} , \] 
where $\sum_{\ell=1}^k\left|\mu_\ell^+-\mu_\ell^-\right|\leq 2$. Thus, we obtain that $r_0^+=r_0^-$ and 
\[\sum_{\ell=1}^k \left(\sqrt{\mu_\ell^+}-\sqrt{\mu_\ell^-}\right)^2\leq \sum_{\ell=1}^k |\mu_\ell^+-\mu_\ell^-|\leq 2 ,\] 
where the first inequality is due to $\left(\sqrt{\mu_\ell^+}-\sqrt{\mu_\ell^-}\right)^2\leq |\mu_\ell^+-\mu_\ell^-|$, for every $\ell$. This concludes the proof.
\end{proof}

Observe that Lemma~\ref{lem:matrix_difference_2} shows that the difference of $M_{ij}^{+} - M_{ij}^{-}$ is small in a certain sense. We can use this information along with Theorem~\ref{thm:eigen_stability_2} to control the distance between the distinct eigenvalues of the matrices $M_{ij}^{+}$ and $M_{ij}^{-}$.

In order to treat the equality case of Proposition~\ref{prop:bound_trees_2} we will also need the following standard result in algebraic graph theory. Let $\epsilon_i$ denote the eccentricity of vertex $i$ in the graph $G$, that is, the maximum distance between any vertex of $G$ and vertex $i$. Recall from Theorem~\ref{thm:neutrino} that the number of zeros of $(\phi^G/\phi^{G\setminus i})$ is the size of the eigenvalue support of $i$.

\begin{lemma}[Lemma 5 in \cite{CoutinhoQuantumSizeGraph}] \label{lem:eccentricty} Let $i$ be a vertex in the graph $G$. Then, the number of zeros of $(\phi^G/\phi^{G\setminus i})$ is at least $\epsilon_i+1$.
\end{lemma}

Finally, we make the crucial observation that strong cospectrality implies that the sets $\Phi_{ij}^{+}$ and $\Phi_{ij}^{-}$ are disjoint. Therefore, if we obtain zeros of $\alpha_{ij}^{+}$ and $\alpha_{ij}^{-}$ close to each other, then we obtain two (distinct) elements of $\Phi_i$ that are close.

We are finally ready for the proof of Proposition~\ref{prop:bound_trees_2}.

\begin{proof}[Proof of Proposition~\ref{prop:bound_trees_2}.]

Recall \eqref{eq:k} and \eqref{eq:alphask}, and consider $M_{ij}^{+} - M_{ij}^{-}$, just like in Lemma~\ref{lem:matrix_difference_2}. Observe that $|\Phi_{ij}^{+}|+|\Phi_{ij}^{-}| > k+1$. Thus, by the pigeonhole principle, there is some index $m$ in $\{1,\dots, k+1\}$, such that $\theta_m(M_{ij}^{+})$ and $\theta_m(M_{ij}^{-})$ are zeros of $\alpha_{ij}^{+}$ and $\alpha_{ij}^{-}$, respectively.

By Lemma~\ref{lem:matrix_difference_2} we have that the difference $M_{ij}^{+} - M_{ij}^{-}$ has precisely two nonzero eigenvalues equal to 
\begin{equation} \label{eq:eigenvalues}
 	\pm\sqrt{\sum_{\ell=1}^k \left(\sqrt{\mu_\ell^+}-\sqrt{\mu_\ell^-}\right)^2},
 \end{equation}
 and moreover we can say that $\sqrt{\sum_{\ell=1}^k \left(\sqrt{\mu_\ell^+}-\sqrt{\mu_\ell^-}\right)^2}\leq \sqrt{2}$. Therefore by Theorem~\ref{thm:eigen_stability_2} we have that

\[\big|\theta_m(M_{ij}^{+})-\theta_m(M_{ij}^{-})\big|\leq \lVert M_{ij}^{+} - M_{ij}^{-}\rVert_{op}\leq \sqrt{2}.\] 

It follows that there are zeros of $\alpha_{ij}^{+}$ and $\alpha_{ij}^{-}$ with distance smaller than, or equal to, $\sqrt{2}$. By the discussion above we conclude that there are two elements of $\Phi_i$ with distance smaller than, or equal to, $\sqrt{2}$. 

We proceed to analyze the equality case.

By Corollary~\ref{cor:partial_fraction_2} we know that $k\geq 1$, because otherwise $\alpha_{ij}^{+}$ and $\alpha_{ij}^{-}$ are equal. We claim that $k=1$, and
\[\alpha_{ij}^{+}(t)=t-r-\dfrac{2}{t-r} \quad \text{and} \quad \alpha_{ij}^{-}(t)=t-r\]
for some real number $r$.

Assume that $|\theta_m(M_{ij}^{+})-\theta_m(M_{ij}^{-})|=\sqrt{2}$. By the equality case in Theorem~\ref{thm:eigen_stability_2} there exists a nonzero vector $\ket v$ spanned by the eigenvectors corresponding to the two nonzero eigenvalues of $M_{ij}^{+} - M_{ij}^{-}$ which is an $m$-th eigenvector of both $M_{ij}^{+}$ and $M_{ij}^{-}$.
Since
\[
	(M_{ij}^{+} - M_{ij}^{-}) \ket v = (\theta_m(M_{ij}^{+})-\theta_m(M_{ij}^{-})) \ket v ,
\]
and $(M_{ij}^{+} - M_{ij}^{-})$ has precisely two nonzero eigenvalues given by \eqref{eq:eigenvalues}, it follows from Lemma~\ref{lem:matrix_difference_2} that 
\[
	2 \leq (\theta_m(M_{ij}^{+})-\theta_m(M_{ij}^{-}))^2 = \sum_{\ell=1}^k \left(\sqrt{\mu_\ell^+}-\sqrt{\mu_\ell^-}\right)^2 \leq \sum_{\ell=1}^k |\mu_\ell^+-\mu_\ell^-|\leq 2,
\]
from which follows that $\left(\sqrt{\mu_\ell^+}-\sqrt{\mu_\ell^-}\right)^2 = |\mu_\ell^+-\mu_\ell^-|$ for every $\ell$. 

Recall from \eqref{eq:alphask} that $\mu_\ell \geq 0$ for all $\ell$. Thus, for every $\ell$, either $\mu_\ell^+=0$ and $\mu_\ell^->0$, or $\mu_\ell^+>0$ and $\mu_\ell^-=0$. The eigenvector $\ket v$ can be easily computed from the expression of $(M_{ij}^{+} - M_{ij}^{-})$, and the equations $M_{ij}^{\pm} \ket v = \theta_m(M_{ij}^{\pm}) \ket v$ are given by
\begin{align*}
	& \left[ {\begin{array}{cccccccc}
	r_0  & \sqrt{\mu_1^\pm} & \sqrt{\mu_2^\pm} & \cdots  & \sqrt{\mu_k^\pm}  \\
    \sqrt{\mu_1^\pm} & r_1 & 0 & \cdots  & 0  \\
    \sqrt{\mu_2^\pm} & 0 & r_2 & \cdots  & 0  \\
	\vdots   & \vdots   & \vdots   &  \ddots  &  \vdots  \\ 
	\sqrt{\mu_k^\pm}   & 0 & 0 & \cdots  & r_k \\ 
	\end{array} } \right]\pmat{ a \\ b(\sqrt{\mu_1^+}-\sqrt{\mu_1^-}) \\ \vdots \\ b(\sqrt{\mu_k^+}-\sqrt{\mu_k^-})} = \hspace{3cm}\\
	& \hspace{7cm} = \theta_m(M_{ij}^{\pm}) \pmat{ a \\ b(\sqrt{\mu_1^+}-\sqrt{\mu_1^-}) \\ \vdots \\ b(\sqrt{\mu_k^+}-\sqrt{\mu_k^-})}.
\end{align*}
Observe that $b = 0$ implies $a = 0$, so we have $b \neq 0$. Thus, for every $\ell$ either
\begin{itemize}
	\item $\mu_\ell^-=0$, $a+br_\ell=\theta_m(M_{ij}^{+})~b$ and $br_\ell=\theta_m(M_{ij}^{-})~b$, and therefore
	\[r_\ell = \theta_m(M_{ij}^{-}) \quad \text{and} \quad \dfrac{a}{b}=\theta_m(M_{ij}^{+})-\theta_m(M_{ij}^{-}); \]
	or
	\item  $\mu_\ell^+=0$, $a+br_\ell=\theta_m(M_{ij}^{-})~b$ and $br_\ell=\theta_m(M_{ij}^{+})~b$, and therefore
	\[r_\ell = \theta_m(M_{ij}^{+}) \quad \text{and} \quad \dfrac{a}{b}=\theta_m(M_{ij}^{-})-\theta_m(M_{ij}^{+}). \]
\end{itemize}

Since $|\theta_m(M_{ij}^{+})-\theta_m(M_{ij}^{-})|=\sqrt{2}$, it cannot happen that $(a/b)$ is simultaneously equal to a positive and negative number, therefore it must be that, for every $\ell$, either $\mu_\ell^-=0$ and $r_\ell = \theta_m(M_{ij}^{-})$; or $\mu_\ell^+=0$ and $r_\ell = \theta_m(M_{ij}^{+})$.

The $r_\ell$'s are by definition distinct, and $\sum_{\ell=1}^k|\mu_\ell^+-\mu_\ell^-|=2$, therefore $k=1$, and, either 
\begin{itemize}
	\item $\mu_1^+=2$, $\mu_1^-=0$ and $r_1=\theta_m(M_{ij}^{-})$, or
	\item $\mu_1^+=0$, $\mu_1^-=2$ and $r_1=\theta_m(M_{ij}^{+})$. 
\end{itemize}

Finally, recall that the principal eigenvector of the adjacency matrix $A$ has positive entries and therefore the largest element of $\Phi_i$ is always in $\Phi_{ij}^+$. A simple analysis concludes that, if we denote by $r=r_0$, it must be that $\alpha_{ij}^+(t)=t-r-\frac{2}{t-r}$ and $\alpha_{ij}^-(t)=t-r$, proving the claims.

It remains to show that the graph $G$ is in fact $P_3$. First, notice that,

\[\dfrac{\phi^{G\setminus i}}{\phi^{G\setminus \{ i,j\}}}(t)=\frac{\alpha_{ij}^+(t)+\alpha_{ij}^-(t)}{2}=t-r-\dfrac{1}{t-r}.\]

By Lemma~\ref{lem:eccentricty} it follows that $j$ has eccentricity $1$ in $G\setminus i$. By symmetry we obtain the analogous conclusion for the vertex $i$. It follows that $i'=j'$ and that the only neighbors of this vertex are $i$ and $j$. 

By Corollary~\ref{cor:cristoffel-darboux} and the definition of $\alpha_{ij}^+$ and $\alpha_{ij}^-$, we have that 
\[\frac{\phi^G}{\phi^{G\setminus \{i,j\}}}=\alpha_{ij}^+\cdot \alpha_{ij}^-,\] 
whence
\[\dfrac{\phi^G}{\phi^{G\setminus i}} = \dfrac{\phi^G}{\phi^{G\setminus \{i, j\}}}\cdot \dfrac{\phi^{G\setminus \{i, j\}}}{\phi^{G\setminus i}}=\alpha_{ij}^+\cdot \alpha_{ij}^-\cdot \dfrac{\phi^{G\setminus \{i, j\}}}{\phi^{G\setminus i}}=\dfrac{(t-r)((t-r)^2-2)}{(t-r)^2-1}.\]

By Lemma~\ref{lem:eccentricty}, $i$ has eccentricity at most $2$ in $G$, from which follows that the only neighbor of $j$ in $G$ is $i'=j'$. By symmetry we obtain that the only neighbor of $i$ in $G$ is also $i'=j'$. We conclude that the graph $G$ consists only of the vertices $i$, $j$ and $i'=j'$.
\end{proof}


\section{Final remarks}\label{sec:problems}

The technology in this paper (together with Lemma $IV.3.2$ in~\cite{BhatiaMatrixAnalysis}) can also be used to provide a proof that if $i$ and $j$ are cospectral and separated by a bridge, then there are two eigenvalues in the support of $i$ whose distance is at most $1$, if and only if the graph is not $P_2$. This is precisely a result obtained in \cite{CoutinhoGodsilJulianovanBommel} with a different technique. 

The attentive reader may have noticed that we did not require $i$ and $j$ to be strongly cospectral in the preceding paragraph. This is the case because a result of Godsil and Smith~\cite[Theorem 9.1]{godsil2017strongly} implies that if $i$ and $j$ are cospectral and $ij$ is a bridge, then $i$ and $j$ are strongly cospectral. So one wonders: is Proposition~\ref{prop:bound_trees_2} also valid assuming only cospectrality? This turns out to be far from true.

There exists a well-known construction of integral trees of arbitrarily large diameter due to Csikvári~\cite{csikvari2010integral}. This construction allows for the spectrum to be specified in advance.

\begin{theorem}[Theorem 11 in \cite{csikvari2010integral}]~\label{thm:csikvari} For every set $S$ of positive integers there is a tree whose positive
eigenvalues are exactly the elements of $S$. If the set $S$ is different from the set $\{1\}$ then the constructed tree will have diameter $2|S|$.
\end{theorem}

By inspecting the proof of Theorem~\ref{thm:csikvari} one can see that if the maximum distance between consecutive elements in $S$ is at least $2$, then the trees have a non-trivial automorphism, thus, in particular, they have cospectral pairs of vertices. It is also easy to guarantee that their distance will be $2$. Putting together these observations we have the following result.

\begin{theorem}\label{thm:bound_tree_fails} There exists a sequence of integral trees $(T_n)$ with cospectral pairs of vertices at distance $2$ such that the minimum distance between elements in $\Phi(T_n)$ tends to $\infty$ as $n\to \infty$.
\end{theorem}

Theorem~\ref{thm:bound_tree_fails} is in sharp contrast to the Proposition~\ref{prop:bound_trees_2}, and shows that the hypothesis of strong cospectrality cannot be reduced to cospectrality.

The results in this paper can also be stated to weighted graphs (with or without loops), in particular, Proposition~\ref{prop:bound_trees_2} holds just the same as long as the bridges have weight $1$ (the equality case will conclude the resulting graph is $P_3$ up to diagonal translations). The weighted generalizations can be easily obtained by replacing all occurrences of $\phi^{G \backslash P}$ in Section~\ref{sec:strategy} by $\rho_P  \cdot \phi^{G \backslash P}$, where $\rho_P$ is the product of the weights in the path. The results in Subsection~\ref{sec:pst} hold for integer weighted graphs. In particular, the technology developed in this paper provides a completely alternative proof to the result that no trees on more than two vertices admit perfect state transfer according to the Laplacian matrix \cite{CoutinhoLiu2}. For reference, we state the most general weighted version we can obtain:

\begin{theorem}\label{thm:bound_general} Let $i$ and $j$ be strongly cospectral vertices in the weighted graph $G$, possibly connected by an edge of weight $a_{i,j}$. Assume that every path from $i$ to $j$, with the exception of the trivial path through the edge $ij$, goes through the neighbors $i_1,\dots, i_u$ and $j_1,\dots, j_v$ of $i$ and $j$, respectively. Then, the minimum distance between elements in $\Phi_i$ is smaller than, or equal to 
\[a_{i,j}+\sqrt{a_{i,j}^2+2\sqrt{\left(\sum_{k=1}^u a_{i,i_k}^2\right)\cdot \left(\sum_{\ell=1}^v a_{j_\ell,j}^2\right)}}.\]
\end{theorem}


\subsection*{Acknowledgements}

All authors acknowledge the support of FAPEMIG. Gabriel Coutinho acknowledges the support of CNPq.

\bibliographystyle{plain}
\IfFileExists{references.bib}
{\bibliography{references.bib}}
{\bibliography{../references}}

\begin{thebibliography}{10}

\bibitem{ImplementationQwalkIBM}
F~Acasiete, F~P Agostini, J~Khatibi Moqadam, and R~Portugal.
\newblock Implementation of quantum walks on ibm quantum computers.
\newblock {\em Quantum Information Processing}, 19:426, 2020.

\bibitem{BhatiaMatrixAnalysis}
Rajendra Bhatia.
\newblock {\em Matrix Analysis}, volume 169.
\newblock Springer, New York, NY, 1997.

\bibitem{ChildsUniversalQComputation}
Andrew~M Childs.
\newblock Universal computation by quantum walk.
\newblock {\em Physical Review Letters}, 102:4,180501, 2009.

\bibitem{ChristandlPSTQuantumSpinNet2}
Matthias Christandl, Nilanjana Datta, Tony~C Dorlas, Artur Ekert, Alastair Kay,
  and Andrew~J Landahl.
\newblock Perfect transfer of arbitrary states in quantum spin networks.
\newblock {\em Physical Review A}, 71:32312, 2005.

\bibitem{CoutinhoPhD}
Gabriel Coutinho.
\newblock {\em Quantum State Transfer in Graphs}.
\newblock PhD thesis, University of Waterloo, 2014.

\bibitem{CoutinhoQuantumSizeGraph}
Gabriel Coutinho.
\newblock Quantum walks and the size of the graph.
\newblock {\em Discrete Mathematics}, 342(10):2765--2769, 2019.

\bibitem{CoutinhoGodsilPSTpolytime}
Gabriel Coutinho and Chris Godsil.
\newblock Perfect state transfer is poly-time.
\newblock {\em Quantum Information \& Computation}, 17:495--502, 2017.

\bibitem{CoutinhoGodsilSurvey}
Gabriel Coutinho and Chris Godsil.
\newblock Continuous-time quantum walks in graphs.
\newblock {\em IMAGE - The Bulletin of the International Linear Algebra
  Society}, Spring:12--21, 2018.

\bibitem{CoutinhoGodsilJulianovanBommel}
Gabriel Coutinho, Chris Godsil, Emanuel Juliano, and Christopher~M van Bommel.
\newblock Quantum walks do not like bridges.
\newblock {\em Linear Algebra and its Applications}, 652:155--172, 2022.

\bibitem{coutinho2022quantum}
Gabriel Coutinho, Chris Godsil, Emanuel Juliano, and Christopher~M van Bommel.
\newblock Quantum walks do not like bridges.
\newblock {\em Linear Algebra and its Applications}, 652:155--172, 2022.

\bibitem{coutinho2022strong}
Gabriel Coutinho, Emanuel Juliano, and Thom{\'a}s~Jung Spier.
\newblock Strong cospectrality in trees.
\newblock {\em arXiv preprint arXiv:2206.02995}, 2022.

\bibitem{coutinho2023decorated}
Gabriel Coutinho, Emanuel Juliano, and Thom{\'a}s~Jung Spier.
\newblock The spectrum of symmetric decorated paths.
\newblock {\em arXiv preprint arXiv:2305.09406}, 2023.

\bibitem{CoutinhoLiu2}
Gabriel Coutinho and Henry Liu.
\newblock No {L}aplacian perfect state transfer in trees.
\newblock {\em SIAM Journal on Discrete Mathematics}, 29:2179--2188, 11 2015.

\bibitem{csikvari2010integral}
P{\'e}ter Csikv{\'a}ri.
\newblock Integral trees of arbitrarily large diameters.
\newblock {\em Journal of Algebraic Combinatorics}, 32(3):371--377, 2010.

\bibitem{fan2013pretty}
Xiaoxia Fan and Chris Godsil.
\newblock Pretty good state transfer on double stars.
\newblock {\em Linear Algebra and Its Applications}, 438(5):2346--2358, 2013.

\bibitem{godsil2017strongly}
Chris Godsil and Jamie Smith.
\newblock Strongly cospectral vertices.
\newblock {\em arXiv preprint arXiv:1709.07975}, 2017.

\bibitem{GodsilAlgebraicCombinatorics}
Chris~D Godsil.
\newblock {\em Algebraic Combinatorics}.
\newblock Chapman \& Hall, 1993.

\bibitem{GodsilStateTransfer12}
Chris~D Godsil.
\newblock State transfer on graphs.
\newblock {\em Discrete Mathematics}, 312:129--147, 2012.

\bibitem{GodsilPerfectStateTransfer12}
Chris~D Godsil.
\newblock When can perfect state transfer occur?
\newblock {\em Electronic Journal of Linear Algebra}, 23:877--890, 2012.

\bibitem{QwalkReview2022}
Karuna Kadian, Sunita Garhwal, and Ajay Kumar.
\newblock Quantum walk and its application domains: A systematic review.
\newblock {\em Computer Science Review}, 41:100419, 2021.

\bibitem{KayLimbo}
Alastair Kay.
\newblock The perfect state transfer graph limbo.
\newblock {\em arXiv preprint arXiv:1808.00696}, 2018.

\end{thebibliography}

	
\end{document}